\documentclass[12pt]{elsarticle}
\graphicspath{ {./figures/} }

\usepackage[a4paper, total={6.5in, 9in}]{geometry}
\usepackage{caption}
\usepackage[hyphens]{url}
\usepackage{hyperref}
\usepackage{amsmath}
\usepackage{amssymb}
\usepackage{amsthm}         
\usepackage{mathtools}      
\usepackage{graphicx}
\usepackage[percent]{overpic} 
\usepackage{empheq}
\usepackage{tcolorbox}
\usepackage{mathrsfs}
\usepackage{pifont}
\usepackage{algorithm}
\usepackage{algpseudocode}
\usepackage{natbib}
\usepackage{cleveref}
\usepackage{lscape}
\usepackage{tabularx}
\usepackage{soul} 
\usepackage[table]{xcolor} 
\usepackage{changepage, threeparttable} 
\usepackage{multirow}
\usepackage{adjustbox}
\usepackage{smartdiagram}
\usepackage{longtable}
\usepackage{enumitem} 
\usepackage{sectsty} 
\usepackage{subcaption} 
\usepackage{booktabs}
\usepackage{rotating} 
\usepackage{bbding}

\usepackage{tikz}
\usepackage{pgfplots}
\pgfplotsset{compat=1.18} 

\newtheorem{theorem}{Theorem}[section]
\newtheorem{proposition}{Proposition}

\theoremstyle{definition}
\newtheorem{definition}{Definition}

\hypersetup{
    colorlinks=false,
    urlcolor=blue,
    linkcolor=blue,
    citecolor=red
}

\begin{document}
\begin{frontmatter}

\title{Convexity and SOS-Convexity of Sum of Separable and Biquadratic Quartic Polynomials and Optimization}

\author[label1]{Shaon Naskar}
\ead{snaskar@isihyd.ac.in}

\author[label1]{Sujeet Kumar Singh\corref{cor1}}
\ead{sksinghma209@gmail.com, sujeet@isihyd.ac.in}

\cortext[cor1]{Corresponding author}

\address[label1]{Indian Statistical Institute, Hyderabad Unit, India, 500007}


\begin{abstract}
Determining whether multivariate polynomials of degree four or higher are nonnegative and convex is a strongly NP-hard problem. To mitigate these computational difficulties, sum-of-squares (SOS) convexity has been proposed as a tractable algebraic relaxation that yields a checkable sufficient condition for convexity and can be expressed as a semidefinite program (SDP). In this work, we introduce a structured subclass of quartic polynomials, called the Sum of Separable and Biquadratic (SPBQ) forms, and conduct a systematic analysis of the connection between convexity and SOS-convexity within this class. Specifically, we show that every convex SPBQ polynomial is necessarily SOS-convex when the associated biquadratic form has size $n \times 2$. We then construct an explicit SPBQ example with a $3 \times 3$ biquadratic form that is convex but fails to be SOS-convex. Finally, we examine both unconstrained and constrained optimization problems involving SPBQ polynomials, demonstrate notable computational benefits compared to general SOS-based methods, and illustrate their use in convex polynomial regression and fluid dynamics.
\end{abstract}


\end{frontmatter}
\makeatletter
\def\ps@pprintTitle{%
 \let\@oddhead\@empty
 \let\@evenhead\@empty
 \def\@oddfoot{}%
 \let\@evenfoot\@oddfoot}
\makeatother

\section{Introduction}
A homogeneous polynomial (also called a form) $f:\mathbb{R}^n\to\mathbb{R}$ is called nonnegative or positive semidefinite (psd) if $f(x)\ge 0$ for all $x\in\mathbb{R}^n$. It is called a sum of squares (sos) if it can be expressed as $f = \sum_{i=1}^m f_i^2$ for some forms $f_1(x), f_2(x), \dots, f_m(x)$. We write $\Sigma_{n,2d}$ for the set of sos forms of degree $2d$ in $n$ variables, and $N_{n,2d}$ for the set of psd forms of degree $2d$ in $n$ variables. In 1888, Hilbert \cite{hilbert1888darstellung} proved that $N_{n,2d}$ and $\Sigma_{n,2d}$ coincide if and only if $n=1$, $d=1$, or $(n,2d)=(3,4)$. Determining when these two sets are equal is a classical problem in real algebraic geometry. A major breakthrough was achieved in 1967 when Motzkin presented the first explicit example of a nonnegative form that is not a sum of squares: the ternary sextic \citep{motzkin1967arithmetic}
\[ M(x,y,z) = x^4 y^2 + x^2 y^4 - 3 x^2 y^2 z^2 + z^6. \]
Since then, many further examples have been found (see, for instance, \cite{choi1977extremal, choi1977old, choi1980real, reznick2000some, robinson1973some}).

The difference between psd polynomials and sum of squares(SOS) polynomials is not just of theoretical interest; it has significant consequences for computation. It is known that deciding whether a multivariate polynomial with degree four or higher is non-negative is an NP-hard problem \cite{murty1987some}. By contrast, checking whether a polynomial is a sum of squares can be posed as a semidefinite programming (SDP) feasibility problem, which can be solved in polynomial time to arbitrary precision \cite{vandenberghe1996semidefinite}. Motivated by the classical fact that all nonnegative univariate ($n=1$) and quadratic ($d=2$) polynomials inherently allow a sum-of-squares (SOS) decomposition, Ahmadi et al. \cite{ahmadi2023sums} have recently proposed a new structured class of polynomials, termed SPQ (sums of separable and quadratic polynomials). Using this framework, they carried out a systematic study of the relationships among nonnegativity, SOS representability, and convexity, and examined a variety of practical applications.
Alongside nonnegativity, polynomial convexity is a fundamental property with far-reaching consequences in both theoretical analysis and practical optimization. Most importantly, convexity can convert optimization problems that are otherwise intractable into ones that are computationally manageable. For instance, although the unconstrained global minimization of a general polynomial is NP-hard \cite{parrilo2001minimizing}, knowing in advance that the polynomial is convex guarantees that every local minimizer is also a global minimizer. As a result, even basic gradient descent methods can be relied upon to converge to the global optimum. The need to certify polynomial convexity also arises in other areas of mathematics; for example, the $d$-th root of a homogeneous, degree-$d$, positive definite polynomial defines a valid norm if and only if the polynomial is convex \cite{reznick2011blenders}. Historically, understanding the computational complexity of testing polynomial convexity was a major challenge, and it was famously highlighted in 1992 as one of seven open problems in complexity theory for numerical optimization \cite{pardalos1992open}. Subsequent work showed that deciding the convexity of polynomials of degree four or higher is strongly NP-hard \cite{ahmadi2013np}. To circumvent these computational difficulties, this paper concentrates on sos-convexity, an algebraic relaxation introduced rigorously by Helton and Nie \cite{helton2010semidefinite}. By requiring the Hessian of the polynomial to admit a sum-of-squares (SOS) decomposition, sos-convexity provides a sound, efficiently checkable sufficient condition for ordinary convexity. Since deciding whether a polynomial is sos-convex can be entirely formulated and solved as a semidefinite program (SDP), it is highly advantageous computationally. In addition, it offers a clean conceptual link between the geometric perspective on convex functions and the algebraic structure of polynomials.
A natural progression from the SPQ family involves elevating the quadratic component to a quartic one. However, since determining the nonnegativity and convexity of general degree four polynomials is known to be strongly NP-hard, we restrict our attention to a specialized class of quartic polynomials known as biquadratic forms. A \textit{biquadratic form} $b(x, y)$ over the variable sets $x = (x_1, \dots, x_n)^{T}$ and $y = (y_1, \dots, y_m)^{T}$ is defined as:
\begin{equation}
b(x, y) = \sum_{i \le j} \sum_{k \le l} \alpha_{ijkl} x_i x_j y_k y_l.
\end{equation}
A defining characteristic of $b(x, y)$ is its structural symmetry: it reduces to a quadratic form in $y$ when $x$ is fixed, and conversely, it becomes a quadratic form in $x$ when $y$ is fixed. Although the nonnegativity and SOS representations of biquadratic forms have been explored in the existing literature \cite{calderon1973note, choi1975positive, el2011class}, our work introduces the \textit{sum of separable and biquadratic(SPBQ)} forms, whose definition is given in \ref{def:SPBQ}. Specifically, this research focuses on investigating the convexity properties of the SPBQ family and exploring their practical applications. Unless otherwise specified, we use the term SPBQ to refer exclusively to SPBQ forms.

\begin{definition}\label{def:SPBQ}
A polynomial 
$p : \mathbb{R}^n \times \mathbb{R}^m \rightarrow \mathbb{R}$ 
is called a \emph{separable plus biquadratic (SPBQ)} form if
\[
p(x, y) = s(x, y) + b(x, y),
\]
where $s(x, y)$ is a separable form, i.e.,
\[
s(x, y) = \sum_{i=1}^n s_i(x_i) + \sum_{j=1}^m t_j(y_j),
\]
for some univariate forms
$s_i, t_j : \mathbb{R} \rightarrow \mathbb{R}$, 
and $b(x, y)$ is a biquadratic form.
\end{definition}


The remainder of the paper is structured as follows. In Section \ref{convex spbq}, we establish that every convex SPBQ polynomial is sos-convex, under the assumption that the biquadratic component of the SPBQ polynomial is of order $n\times 2$. In Section \ref{spbq counter}, we provide an example of a convex SPBQ polynomial that is not sos-convex, where the associated biquadratic part has order $3\times 3$ and is itself sos. Section \ref{optimization} is devoted to the study of convex SPBQ polynomial optimization. Specifically, Section \ref{uncons opt} examines unconstrained SPBQ polynomial optimization and shows that our proposed SPBQ techniques offer a computational time advantage over conventional sos-based methods, and Section \ref{cons opt} deals with that of constrained SPBQ polynomial optimization. We conclude by presenting two applications of SPBQ polynomials in convex polynomial regression and fluid dynamics, which are discussed in detail in Section \ref{regression app.} and Section \ref{fluid app.}, respectively.

\section{When Convexity of SPBQ polynomials implies SOS-Convexity?}\label{convex spbq}
We now state our main result, which establishes that convexity and sos-convexity are equivalent for SPBQ polynomials. In particular, we identify a specific subclass of SPBQ polynomials whose convexity guaranties that every member of this class is sos-convex.
\begin{theorem}
 Consider a biquadratic form \(b(x,y)\) of bidegree \( (n,2) \), and a nonnegative separable form in the variables \((x_1, x_2, \dots, x_n, y_1, y_2)\). Define the polynomial \(f\) by
\[
f(x,y) = \sum_{i=1}^n \alpha_i x_i^4 + \sum_{j=1}^2 \beta_j y_j^4 + b(x,y),
\]
where the coefficients satisfy \(\alpha_i \geq 0\) and \(\beta_j \geq 0\) for all \(1 \le i \le n\) and \(1 \le j \le 2\). Then the form \(f\) is convex if and only if it is sos-convex, i.e., the Hessian of \(f\) admits a sum-of-squares factorization.
\end{theorem}

\begin{proof}
We consider the homogeneous family of separable plus biquadratic polynomials defined as follows:
\begin{equation}\label{spbq}
f(x,y) = \sum_{i=1}^n \alpha_i x_i^4 + \sum_{j=1}^2 \beta_j y_j^4 + b(x,y),
\end{equation}
where $x\in \mathbb{R}^n, y\in\mathbb{R}^{2}$ and $b(x,y)$ is a nonnegative biquadratic polynomial of order $n\times 2$, with $\alpha_i\geq 0, \beta_j\geq 0, \forall i,j$.

We use $H_f (x,y)$ to denote the Hessian of $f$. The Hessian matrix $H_f$ is an $(n+2)\times(n+2)$ matrix whose entries are all quadratic forms. The polynomial $f$ is convex if and only if the polynomial $z^T H_f(x,y)z$ is psd, where $z=(z_x^T,z_y^T)^T$, $z_x\in\mathbb{R}^n$ and $z_y\in\mathbb{R}^2$. We denote the Hessian of the separable part by $H_{sep}(x,y)$ and the Hessian of $b(x,y)$ by $H_b(x,y)$. Hence $H_f(x,y) = H_{sep}(x,y) + H_b(x,y)$. First, we consider the structure of $H_b(x,y)$.

Observe that if we define
\[
[P(x)]_{ij}=\frac{\partial^2 b(x,y)}{\partial y_i \partial y_j},
\]
then the matrix $P(x)$ depends only on $x$. Moreover,
\begin{equation}\label{P(x)}
\frac{1}{2}y^T P(x) y = b(x,y). 
\end{equation}

Similarly, define
\[
[Q(y)]_{ij}=\frac{\partial^2 b(x,y)}{\partial x_i \partial x_j}.
\]
Then $Q(y)$ depends only on $y$, and we have
\begin{equation}\label{Q(y)}
\frac{1}{2}x^T Q(y) x = b(x,y). 
\end{equation}

Putting these together, the Hessian of $b(x,y)$ can be written as
\begin{equation}\label{H_b}
H_b(x,y)=
\begin{pmatrix}
Q(y) & R(x,y) \\
R^T(x,y) & P(x)
\end{pmatrix}.
\end{equation}
In general, the matrix \(R(x,y)\) is not symmetric. The entries of \(R(x,y)\) are square-free monomials, each of which is a multiple of \(x_i y_j\) for some indices \(i,j\) satisfying \(1 \le i \le n\) and \(1 \le j \le 2\).

The Hessian \(H_{\mathrm{sep}}(x,y)\) associated with the separable component is given by:
\begin{equation}\label{H_sep}
    \begin{pmatrix}
        12\alpha_1x_1^2 & 0 & \dots & 0 & 0 & 0\\
        0 & 12\alpha_2x_2^2 & \dots & 0 & 0 & 0\\
        \vdots & \vdots & \ddots & \vdots & \vdots & \vdots\\
        0 & 0 & \dots & 12\alpha_n x_n^2 & 0 & 0\\
        0 & 0 & \dots & 0 & 12\beta_1y_1^2 & 0\\
        0 & 0 & \dots & 0 & 0 & 12\beta_2y_2^2\\
    \end{pmatrix}.
\end{equation}

Now we show that if $b(x,y)$ is not positive semidefinite, then $f(x,y)$ is not convex. 
Indeed, suppose that there exists $x_0\in\mathbb{R}^n, y_0\in\mathbb{R}^2$ such that
\[
b(x_0,y_0) < 0 .
\]
Consider a vector $z$ with components $z_x=0$ and $z_y=y_0$. Then
\[
z^T H(x,y) z 
\bigg|_{z_x=0,\; x=x_0,\; y=0,\; z_y=y_0}
= y_0^T P(x_0) y_0
= 2b(x_0,y_0) < 0 .
\]
Hence, the Hessian has a negative quadratic form in this direction, which implies that $f(x,y)$ is not convex.

On the other hand, suppose that $b(x,y)$ is psd; we will prove that $z^T H(x,y) z$ is psd and hence $f(x,y)$ is convex. Moreover, we will show that whenever $z^T H(x,y) z$ is psd it will be sos and hence $f(x,y)$ is sos-convex.

By Calder\'{o}n's Theorem \cite{calderon1973note}, since $y \in \mathbb{R}^2$, any non-negative biquadratic form $b(x,y)$ can be exactly decomposed into a sum of squares of bilinear forms:
\begin{equation}
    b(x,y) = \sum_{k=1}^m B_k(x,y)^2
\end{equation}
where $B_k(x,y) = \sum_{i=1}^n (c_{k,i} x_i y_1 + d_{k,i} x_i y_2)$ for some real coefficients $c_{k,i}$ and $d_{k,i}$.

Now, observe that $B_k(x,y)$ can also be expressed as 
\[
B_k(x,y) = (x;y)^TM_k(x;y)
\]
where $(x;y)^T =(x_1,x_2,\dots,x_n,y_1,y_2)$ and $M_k$ is a scalar matrix of order $(n+2)\times (n+2)$. Furthermore,
\[
\begin{aligned}
    &\frac{\partial^2 }{\partial x_i^2}B_k^2(x,y) = 2\left(\frac{\partial}{\partial x_i}((x;y)^TM_k(x;y))\right)^2\\
    &\frac{\partial^2 }{\partial x_i\partial x_j}B_k^2(x,y) = 2\left(\frac{\partial}{\partial x_i}( (x;y)^TM_k(x;y))\right)\left(\frac{\partial}{\partial x_j}( (x;y)^TM_k(x;y))\right)\\
   &\frac{\partial^2 }{\partial y_i^2}B_k^2(x,y) = 2\left(\frac{\partial}{\partial y_i}( (x;y)^TM_k(x;y))\right)^2\\
    &\frac{\partial^2 }{\partial y_i\partial y_j}B_k^2(x,y) = 2\left(\frac{\partial}{\partial y_i}( (x;y)^TM_k(x;y))\right)\left(\frac{\partial}{\partial y_j}( (x;y)^TM_k(x;y))\right)\\
    & \frac{\partial^2 }{\partial x_i\partial y_j}B_k^2(x,y) = 2\left(\frac{\partial}{\partial x_i}((x;y)^TM_k(x;y))\right)\left(\frac{\partial}{\partial y_j}( (x;y)^TM_k(x;y))\right) + 2s\big((x;y)^TM_k(x;y)\big)
\end{aligned}
\]
So, the Hessian of $b(x,y)$ can be rewritten as
\begin{equation}\label{H_bnew}
    H_b(x,y) = \sum_{k=1}^m H_{b_k} + \sum_{\text{cross}=1}^m H_{b_{\text{cross}}} ,
\end{equation}
where $H_{b_k} = 2G_{b_k}G_{b_k}^T$,
\[
G_{b_k}=\begin{pmatrix}
    \frac{\partial}{\partial x_1}((x;y)^TM_k(x;y))\\
    \frac{\partial}{\partial x_2}((x;y)^TM_k(x;y))\\
    \vdots\\
     \frac{\partial}{\partial x_n}((x;y)^TM_k(x;y))\\
      \frac{\partial}{\partial y_1}((x;y)^TM_k(x;y))\\
       \frac{\partial}{\partial y_2}((x;y)^TM_k(x;y))
\end{pmatrix}
\]
and the matrix $H_{b_{\text{cross}}}$ contains square-free terms, like $x_iy_j$, where $1\leq i\leq n, 1\leq j\leq 2$.
So,
\[
    z^TH_bz = 2\sum_{k=1}^m(z^TG_{b_k})^2 + 2\sum_{\text{cross}=1}^m(z^TH_{b_{\text{cross}}}z) 
\]
and
\begin{equation}\label{sos-con}
    z^TH_fz = \sum_{i=1}^n 12\alpha_i x_i^2z_{x_i}^2 + \sum_{j=1}^2 12\beta_j y_j^2z_{y_j}^2 + 2\sum_{k=1}^m(z^TG_{b_k})^2 + 2\sum_{\text{cross}=1}^m(z^TH_{b_{\text{cross}}}z).
\end{equation}

Clearly, the first three terms of equation \eqref{sos-con} are sos and the last term is a square-free term. Before going to the proof, we first give a summary of the intuition behind the rest of the proof. It is obvious that if the last term of \eqref{sos-con}, i.e., $2\sum_{\text{cross}=1}^m(z^TH_{b_{\text{cross}}}z)$, can be paired with the remaining three sos terms to form a new sos term, then $z^TH_fz$ will be sos. However, when at least one term of $2\sum_{\text{cross}=1}^m(z^TH_{b_{\text{cross}}}z)$ remains unpaired, then $z^TH_fz$ will not be sos. In that case, it may happen that $z^TH_fz$ is psd but not sos. We know that if a polynomial is convex, then each principal minor of the corresponding Hessian matrix of the polynomial should be nonnegative. But if any principal minor of the Hessian matrix fails to be psd, then the polynomial will not be convex. We will show that if there is any square-free term remaining, then $z^TH_fz$ will not be psd and hence $f$ will not be convex.

We shall now provide a formal proof of the theorem. The term $\sum_{\text{cross}=1}^m(z^TH_{b_{\text{cross}}}z)$ contains four different types of elements, i.e., $a_1x_iy_iz_{x_i}z_{y_i}$, $a_2x_iy_jz_{x_i}z_{y_t}(j\neq t)$, $a_3x_iy_jz_{x_s}z_{y_j}(i\neq s)$, and $a_4x_iy_jz_{x_s}z_{y_t}(i\neq j\neq s\neq t)$, where $a_1,a_2,a_3,a_4$ are real constants.

Observe that each principal minor of order $1\times 1$ of the Hessian matrix $H_f(x,y)$ will be $12\alpha_ix_i^2 + Q_{ii}(y)$ for $i=1,2,\dots,n$, and $12\beta_jy_j^2 + P_{jj}(x)$ for $j=1,2$. Following equations \eqref{P(x)} and \eqref{Q(y)}, and since $b(x,y)$ is sos, we can say that $Q_{ii}(y)\geq 0$ and $P_{jj}(x)\geq 0$, $\forall i,j$. Now from (Th. $2.2$, \cite{ahmadi2023sums}) we can say that all the principal minors of order $1\times 1$ of the Hessian matrix $H_f(x,y)$ are sos.

Now we will consider the principal minors of order $2\times 2$ of $H_f$ and show that whenever at least one element of $\sum_{\text{cross}=1}^m(z^TH_{b_{\text{cross}}}z)$ remains unpaired, the corresponding principal minor will not be psd and therefore the polynomial $f$ will not be convex.

First consider if the element $a_1x_iy_jz_{x_i}z_{y_j}$ remains unpaired in equation \eqref{sos-con}. We will consider the following $2\times 2$ minor:
\[
\begin{aligned}
A_1&=\begin{vmatrix}
    \frac{\partial^2}{\partial x_i^2}(f) & \frac{\partial^2}{\partial x_i \partial y_j}(f)\\
    \frac{\partial^2}{\partial x_i \partial y_j}(f) & \frac{\partial^2}{\partial y_j^2}(f)
\end{vmatrix}\\
& = \begin{vmatrix}
    12\alpha_i x_i^2 + Q_{x_{i}x_{i}}(y_1,y_2) & R_{x_{i}y_{j}}(x,y)\\
    R_{x_{i}y_{j}}^T(x,y) & 12\beta_j y_j^2 + P_{y_{j}y_{j}}(x_1,x_2,\dots,x_n)
\end{vmatrix}
\end{aligned}
\]
Observe that the elements of $Q_{x_{i}x_{i}}, P_{y_{j}y_{j}}$ and $R_{x_{i}y_{j}}$ come from $\sum_{k=1}^m H_{b_k} + \sum_{\text{cross}=1}^m H_{b_{\text{cross}}}$. Since the term $a_1x_iy_jz_{x_i}z_{y_j}$ remains unpaired, except for $x_i$, set all the components of $x$ equal to $0$ and let $x_i=1$. Without loss of generality, set $y_2 =0$ and $y_1=1$. So $(\text{coefficient of } Q_{x_i x_i} +12\alpha_i)(\text{coefficient of } P_{y_1 y_1} +12\beta_1) < \big(\text{sum of coefficients of } R_{x_i y_1}(x_i,y_1)\big)^2$. Then the minor will become:
\[
A_1=\begin{vmatrix}
     12\alpha_i  + Q_{x_{i}x_{i}} & R_{x_{i}y_{1}}\\
     R_{x_{i}y_{1}}^T & 12\beta_1 + P_{y_{1}y_{1}}
\end{vmatrix}
\]
Therefore, $\det(A_1)$ is a scalar, and $\det(A_1) < 0$. Thus, in this case, $f$ will not be convex.

Next, consider if the element $a_2x_iy_jz_{x_i}z_{y_t}(j\neq t)$ remains unpaired in equation \eqref{sos-con}. We consider the following $2\times 2$ minor:
\[
\begin{aligned}
A_2&=\begin{vmatrix}
    \frac{\partial^2}{\partial x_i^2}(f) & \frac{\partial^2}{\partial x_i \partial y_t}(f)\\
    \frac{\partial^2}{\partial x_i \partial y_t}(f) & \frac{\partial^2}{\partial y_t^2}(f)
\end{vmatrix}\\
& = \begin{vmatrix}
    12\alpha_i x_i^2 + Q_{x_{i}x_{i}}(y_1,y_2) & R_{x_{i}y_{t}}(x,y)\\
    R_{x_{i}y_{t}}^T(x,y) & 12\beta_t y_t^2 + P_{y_{t}y_{t}}(x_1,x_2,\dots,x_n)
\end{vmatrix}
\end{aligned}
\]
Since the term $a_2x_iy_jz_{x_i}z_{y_t}$ remains unpaired, except for $x_i$, set all the components of $x$ equal to $0$ and let $x_i=1$. Without loss of generality, set $y_t = y_2 =0$. Then the minors become
\[
A_2=\begin{vmatrix}
    12\alpha_i  + Q_{x_{i}x_{i}}(y_1) & R_{x_{i}y_{2}}(y_1)\\
    R_{x_{i}y_{2}}^T(y_1) &  P_{y_{2}y_{2}}
\end{vmatrix}
\]
and $\det(A_2)$ be a quadratic function of the form $l - my_1^2$ and this takes negative values when $y_1 > \pm\sqrt{\frac{l}{m}}$. Therefore, $f$ is also non-convex in this case.

For the remaining two elements $a_3x_iy_jz_{x_s}z_{y_j}(i\neq s)$ and $a_4x_iy_jz_{x_s}z_{y_t}(i\neq j\neq s\neq t)$, similar arguments show that $f$ is not convex.

When these four elements appear as a linear combination in equation \eqref{sos-con}, we can choose any one of them and proceed similarly, showing that $z^TH_f z$ will not be psd and therefore $f$ will not be convex. This completes the proof.
\end{proof}

\section{A Convex SPBQ polynomial which is not SOS-Convex}\label{spbq counter}

In this section, we introduce a polynomial that is convex but not sos-convex. The earliest example of such a polynomial was provided by Ahmadi et al. \citep{ahmadi2012convex}. For biquadratic polynomials, Choi was the first to construct a nonnegative yet not sos biquadratic polynomial of minimal order \cite{choi1975positive}. In [Theorem $2.3$, \cite{ahmadi2013np}], the authors describe a method to obtain convex but not sos-convex polynomials starting from nonnegative but not sos biquadratic polynomials. In contrast, our construction yields a convex but not sos-convex polynomial by using an sos biquadratic polynomial of size $3\times 3$ together with a separable term, as detailed below:

\[
    B^*(x,y) = \sum_{i=1}^3 \Big( 2 x_i^2 y_i^2 - 2 x_i^2 y_i y_{i+1} - 2 x_i x_{i+1} y_i^2 + 3 x_i^2 y_{i+1}^2 + 3 x_{i+1}^2 y_i^2 \Big)
\]
\[
    f_{sep}(x,y) = \sum_{i=1}^3 \left(x_i^4 + y_i^4\right) 
    \]

The target polynomial is parameterized by $\gamma > 0$ as $f(x,y) = f_{sep}(x,y) + \gamma B^*(x,y)$. For our counterexample, we fix $\gamma = 0.1$. Our example was derived using sum-of-squares (SOS) optimization software.

\begin{theorem}
    The polynomial $f(x,y) = f_{sep}(x,y) + 0.1 B^*(x,y)$ is convex.
\end{theorem}

\begin{proof}
    Let $H(x,y)$ be the Hessian matrix of $f(x,y)$. We show that $f(x,y)$ is convex from the fact that
    \[
    (x_1^2+x_2^2 +x_3^2 + y_1^2+y_2^2 + y_3^2)H(x,y)=N^T(x,y)N(x,y)
    \]
for some polynomial matrix $N(x,y)$, or equivalently,
\begin{equation}\label{counter}
(x_1^2+x_2^2 +x_3^2 + y_1^2+y_2^2 + y_3^2)z^TH(x,y)z
\end{equation}
is an sos in $\mathbb{R}[x,y,z]$, which shows that $H(x,y)$ is a PSD polynomial matrix. We provide an explicit sos representation in terms of rational Gram Matrices for the polynomial \eqref{counter} in Appendix A \ref{Appendix:A}.
\end{proof}

\begin{theorem}
    The polynomial $f(x,y) = f_{sep}(x,y) + 0.1 B^*(x,y)$ is not sos-convex.
\end{theorem}

\begin{proof}

Suppose 
\[
p = z^T \nabla^2 f z = \sum_i p_i^2,
\]
where each $p_i(x,y)$ is a bilinear form. We prove that the polynomial $z^T \nabla^2 f z$ in 12 variables $(x_1,x_2,x_3,y_1,y_2,y_3,z_{x_1},z_{x_2},z_{x_3},z_{y_1},z_{y_2},z_{y_3})$ of degree $4$ is not sos. The explicit expression of $z^T \nabla^2 f z$ is given in Appendix B \ref{Appendix:B}.

Observe that the terms $x_i^2 z_{x_j}^{2}$, where $i\neq j \in \{1,2,3\}$, and $y_m^2 z_{y_n}^{2}$, where $m\neq n \in \{1,2,3\}$ are absent in $p$, which means they are also absent in $p_i^2$, and hence the terms are absent in $p_i$. Rewrite $p_i = s_i + t_i$, where $s_i$ contains only the terms $x_i z_{x_i}, y_i z_{y_i}$ for $i=1,2,3$, and $t_i$ contains only the terms $x_i z_{y_j}, y_iz_{x_j}$ for $i,j =1,2,3$.

Now, equating the coefficients of $p = \sum_i p_i^2 = \sum_i (s_i + t_i)^2$ leads to the following:

\[
    2s_i t_i =0
\]

\begin{align*}
    s_i^2 =\; &+ 0.4\Big[x_1z_{x_1} + y_2z_{y_2}\Big]^2
+ 0.4\Big[x_2z_{x_2} + y_3z_{y_3}\Big]^2
+ 0.4\Big[x_3z_{x_3} + y_1z_{y_1}\Big]^2 \\
&+ 0.4\Big[x_2z_{x_2} + y_1z_{y_1}\Big]^2
+ 0.4\Big[x_3z_{x_3} + y_2z_{y_2}\Big]^2
+ 0.4\Big[x_1z_{x_1} + y_3z_{y_3}\Big]^2 \\
&+ 0.4\Big[x_1z_{x_1} + y_2z_{y_2}\Big]^2
+ 0.4\Big[x_2z_{x_2} + y_3z_{y_3}\Big]^2
+ 0.4\Big[x_3z_{x_3} + y_1z_{y_1}\Big]^2 \\
&+ 0.4\Big[x_2z_{x_2} + y_1z_{y_1}\Big]^2
+ 0.4\Big[x_3z_{x_3} + y_2z_{y_2}\Big]^2
+ 0.4\Big[x_1z_{x_1} + y_3z_{y_3}\Big]^2 \\
&+ 0.2\Big[x_1z_{x_1} + y_1z_{y_1}\Big]^2
+ 0.2\Big[x_1z_{x_1} + y_2z_{y_2}\Big]^2
+ 0.2\Big[x_2z_{x_2} + y_2z_{y_2}\Big]^2 \\
&+ 0.2\Big[x_2z_{x_2} + y_3z_{y_3}\Big]^2
+ 0.2\Big[x_3z_{x_3} + y_3z_{y_3}\Big]^2
+ 0.2\Big[x_3z_{x_3} + y_1z_{y_1}\Big]^2 \\
&+ 0.2\Big[x_1z_{x_1} + y_1z_{y_1}\Big]^2
+ 0.2\Big[x_2z_{x_2} + y_1z_{y_1}\Big]^2
+ 0.2\Big[x_2z_{x_2} + y_2z_{y_2}\Big]^2 \\
&+ 0.2\Big[x_3z_{x_3} + y_2z_{y_2}\Big]^2
+ 0.2\Big[x_3z_{x_3} + y_3z_{y_3}\Big]^2
+ 0.2\Big[x_1z_{x_1} + y_3z_{y_3}\Big]^2 \\
&+ 9.6\Big(x_1^2z_{x_1}^2 + x_2^2z_{x_2}^2 + x_3^2z_{x_3}^2
+ y_1^2z_{y_1}^2 + y_2^2z_{y_2}^2 + y_3^2z_{y_3}^2\Big).
    \end{align*}

\begin{align*}
     t_i^2 =\; & 0.2\Big[(y_1 - y_2)z_{x_1} + x_1(z_{y_1} - z_{y_2})\Big]^2 \\
&+ 0.2\Big[(y_2 - y_3)z_{x_2} + x_2(z_{y_2} - z_{y_3})\Big]^2 \\
&+ 0.2\Big[(y_3 - y_1)z_{x_3} + x_3(z_{y_3} - z_{y_1})\Big]^2 \\
&+ 0.2\Big[(x_1 - x_2)z_{y_1} + y_1(z_{x_1} - z_{x_2})\Big]^2 \\
&+ 0.2\Big[(x_2 - x_3)z_{y_2} + y_2(z_{x_2} - z_{x_3})\Big]^2 \\
&+ 0.2\Big[(x_3 - x_1)z_{y_3} + y_3(z_{x_3} - z_{x_1})\Big]^2 \\
&+ 0.4y_2^2z_{x_1}^2 + 0.4x_1^2z_{y_2}^2
+ 0.4y_3^2z_{x_2}^2 + 0.4x_2^2z_{y_3}^2 \\
&+ 0.4y_1^2z_{x_3}^2 + 0.4x_3^2z_{y_1}^2
+ 0.4y_1^2z_{x_2}^2 + 0.4x_2^2z_{y_1}^2 \\
&+ 0.4y_2^2z_{x_3}^2 + 0.4x_3^2z_{y_2}^2
+ 0.4y_3^2z_{x_1}^2 + 0.4x_1^2z_{y_3}^2 \\
&- 0.4x_1y_2z_{x_1}z_{y_1}
- 0.4x_1y_1z_{x_1}z_{y_2}
- 0.4x_2y_3z_{x_2}z_{y_2}
- 0.4x_2y_2z_{x_2}z_{y_3} \\
&- 0.4x_3y_1z_{x_3}z_{y_3}
- 0.4x_3y_3z_{x_3}z_{y_1}
- 0.4x_1y_1z_{x_2}z_{y_1}
- 0.4x_2y_1z_{x_1}z_{y_1} \\
&- 0.4x_2y_2z_{x_3}z_{y_2}
- 0.4x_3y_2z_{x_2}z_{y_2}
- 0.4x_3y_3z_{x_1}z_{y_3}
- 0.4x_1y_3z_{x_3}z_{y_3}.
\end{align*}
The last equality is absurd, as the right hand side takes a negative value when evaluated at:
\[
(x_1 =1, x_2=0, x_3=2, y_1=1, y_2=0, y_3=-2, z_{x_1}=-1, z_{x_2}=0, z_{x_3}=2, z_{y_1}=-1, z_{y_2}=0, \\z_{y_3}=2).
\]

Therefore, we have established that $p=z^T \nabla^2 f z$ cannot be expressed as the sum of squares of bilinear forms.

\end{proof}

\section{Convex Optimization with SPBQ Polynomials}\label{optimization}

\subsection{Unconstrained Case}\label{uncons opt}
Let us consider the following unconstrained polynomial optimization problem to find the global lower bound $p^*$ and recover a guaranteed optimal solution \cite{belousov2002frank}:
\begin{equation}\label{non-neg}
\begin{aligned}
    p^* = &\sup_{\gamma\in \mathbb{R}} \gamma\\
    & \text{s.t.} \quad p(x) - \gamma \geq 0
    \end{aligned}
\end{equation}

The nonnegativity constraint in \eqref{non-neg} can be replaced with a sum-of-squares (SOS) constraint, allowing the lower bound to be computed via a single semidefinite program:
\begin{equation}\label{semi-sos}
\begin{aligned}
    p^{sos} = &\sup_{\gamma\in \mathbb{R}} \gamma\\
    & \text{s.t.}\quad  p(x) - \gamma \quad \text{is sos}
    \end{aligned}
\end{equation}

The polynomial $p(x) - \gamma$ remains convex for any scalar $\gamma$ provided that $p(x)$ itself is convex. It is well-established that $p^*=p^{sos}$ when all nonnegative convex polynomials are SOS; however, this equivalence does not hold in general \cite{blekherman2009convex}. Remarkably, $p^* = p^{sos}$ is the same exact as when $p(x)$ is a convex SPBQ polynomial whose biquadratic component is of order $n\times 2$. Therefore, we can formulate the following highly structured optimization problem:
\begin{equation}\label{Semi-SPBQ}
    \begin{aligned}
        p^* = &\sup_{\gamma\in \mathbb{R}} \gamma\\
    & \text{s.t.}\quad  p(x) - \gamma \in \Sigma_{\text{SPBQ}}
    \end{aligned}
\end{equation}
where $\Sigma_{\text{SPBQ}}$ denotes the set of convex SPBQ polynomials with biquadratic forms of order $n\times 2$.

The bounds derived from the semidefinite programs \eqref{semi-sos} and \eqref{Semi-SPBQ} coincide (in fact, they are zero in all instances, since we work with homogeneous polynomials). The key difference lies in the size of the resulting semidefinite programs, with \eqref{Semi-SPBQ} being substantially more compact. Consider a homogeneous convex SPBQ polynomial $p(x)$ in $(n+2)$ variables of degree $4$. The semidefinite constraint associated with \eqref{semi-sos} has dimension $\binom{n+3}{2}\times \binom{n+3}{2}$. By contrast, the formulation in \eqref{Semi-SPBQ} only requires univariate homogeneous SOS polynomials of degree $4$ together with a biquadratic SOS form of size $n\times 2$. Since biquadratic forms contain only mixed terms of the form $x_ix_jy_ky_l$, the dimension of the largest semidefinite constraint is dramatically reduced to $\max\{\binom{2}{1}\times \binom{2}{1},\,2n\times 2n\}$.

To illustrate the practical scale of this theoretical improvement, we compare the two formulations on randomly generated polynomials; the results are reported in Table \ref{tab:computation_times}. The numerical experiments are implemented in MATLAB using the modeling framework YALMIP \cite{lofberg2008modeling} and the interior-point solver MOSEK \cite{mosek_toolbox_9_1}. All runs are performed on a computer with a $2.81$ GHz processor and $16$ GB of RAM.

\begin{table}[htpb]
    \centering
    \renewcommand{\arraystretch}{1.3}
    \begin{tabular}{||c|c|c|c|c||}
        \hline
        $(n, m)$ & Variables ($n+m$) & General SOS Time (s) & SPBQ Time (s) & Speedup Factor \\
        \hline\hline
        $(4, 2)$  & $6$  & $0.91$ & $0.21$ & $4.33$x \\
        \hline
        $(10, 2)$ & $12$ & $20.54$ & $0.90$ & $22.8$x \\
        \hline
        $(20, 2)$ & $22$ & $76.08$ & $1.99$ & $38.23$x \\
        \hline
        $(30, 2)$ & $32$ & $893.1$      & $2.67$ & $334.49$x \\
        \hline
        $(40, 2)$ & $42$ & Out of memory     & $5.55$ & --- \\
        \hline
        $(50, 2)$ & $52$ & Out of memory     & $11.72$ & --- \\
        \hline
        $(60, 2)$ & $62$ & Out of memory     & $23.57$ & --- \\
        \hline
        $(80, 2)$ & $82$ & Out of memory     & $87.14$ & --- \\
        \hline
    \end{tabular}
    \caption{Computation times (in seconds) for the SOS method and the proposed SPBQ method}
    \label{tab:computation_times}
\end{table}

\subsection{Constrained Case}\label{cons opt}

Let us consider a constrained polynomial optimization problem (POP) formulated as
\begin{equation}\label{cons}
\begin{aligned}
\inf_{x \in \mathbb{R}^n} \quad & p(x) \\
\text{s.t.} \quad & x \in K,
\end{aligned}
\end{equation}
where the feasible region $K$ is defined by a set of polynomial inequalities,
\[
K := \left\{ x \in \mathbb{R}^n \mid g_i(x) \le 0, \; \text{for } i = 1, \dots, m \right\}.
\]

It is well-established that finding the optimal value $p^*$ of this problem is equivalent to maximizing a scalar lower bound $\gamma$, that is,
\begin{equation}
\begin{aligned}
p^* = \sup_{\gamma \in \mathbb{R}} \quad & \gamma \\
\text{s.t.} \quad & p(x) - \gamma \ge 0, \quad \forall x \in K.
\end{aligned}
\end{equation}

To address such problems computationally, Lasserre \cite{lasserre2001global} introduced a hierarchy of semidefinite programming (SDP) relaxations. Under suitable assumptions on the set $K$, this hierarchy yields a sequence of lower bounds that converges asymptotically to $p^*$. Subsequent work has thoroughly examined the finite convergence behavior of this hierarchy, especially for convex POPs \cite{de2011lasserre, lasserre2009convexity}.

In particular, when the Slater regularity condition is satisfied and the polynomials $p(x)$ and $g_i(x)$ are sos-convex, the Lasserre hierarchy attains exact convergence at the first relaxation level.\cite{lasserre2009convexity}.

\begin{proposition}\label{prop:spbq_opt}
Consider the polynomial optimization problem specified in \eqref{cons}, and assume that the constraint functions $g_1(x), \dots, g_m(x)$ satisfy the Slater regularity condition; that is, there exists a strictly feasible point $\bar{x} \in \mathbb{R}^n$ such that $g_i(\bar{x}) < 0$ for all $i \in \{1, \dots, m\}$. 

If the objective function $p(x)$ and the constraint functions $g_i(x)$ are all convex SPBQ forms, then the exact optimal value $p^*$ can be computed by solving the corresponding semidefinite programming relaxation:
\begin{equation}\label{cons spbq}
\begin{aligned}
&p^* = \sup_{\gamma \in \mathbb{R}, \; y \in \mathbb{R}^m} \quad  \gamma \\
\text{s.t.} \quad 
& p(x) - \gamma + \sum_{i=1}^m y_i g_i(x) \in \Sigma_{\text{SPBQ}}, \\
& y \ge 0.
\end{aligned}
\end{equation}
\end{proposition}

\begin{proof}
    Applying the convex Farkas lemma (see, e.g., \cite{stoer2012convexity}), the presence of the Slater regularity condition ensures that a real scalar $\gamma$ constitutes a valid lower bound for the optimum $p^*$ of \eqref{cons} if and only if there exists a vector of dual multipliers $y = (y_1, \dots, y_m)^T \ge 0$ such that the expression
\[
p(x) - \gamma + \sum_{i=1}^m y_i g_i(x)
\]
is globally nonnegative. Therefore, we can rewrite the optimal value as follows:
\begin{equation}
\begin{aligned}
&p^* = \sup_{\gamma \in \mathbb{R},\, y \in \mathbb{R}^m} \quad  \gamma \\
\text{s.t.} \quad 
& p(x) - \gamma + \sum_{i=1}^m y_i g_i(x) \in N_{n+2,4}, \\
& y \ge 0.
\end{aligned}
\end{equation}

By our initial assumption, the objective function $p(x)$ and all constraint functions $g_1(x), \dots, g_m(x)$ are convex SPBQ forms. Because $y \ge 0$, any nonnegative linear combination of these functions, namely
\[
p(x) - \gamma + \sum_{i=1}^m y_i g_i(x),
\]
remains a convex SPBQ polynomial. 

This resulting expression is nonnegative if and only if it belongs to the set $\Sigma_{\text{SPBQ}}$. As a consequence, the exact optimal value $p^*$ can be reliably computed by solving the corresponding semidefinite program presented in \eqref{cons spbq}.
\end{proof}

\section{Convex SPBQ Polynomial Regression}\label{regression app.}

Convex regression is a fundamental problem in statistics and machine learning that focuses on recovering an unknown convex function 
$f: \mathbb{R}^n \rightarrow \mathbb{R}$ from $m$ noisy feature–response samples 
$\{(\mathbf{x}_i, y_i)\}_{i=1}^m$. 
To make predictions at new, unseen inputs, one seeks to fit a polynomial regressor 
$p$ of a chosen degree that faithfully reflects the structure of the data. 
This problem is commonly posed as minimizing the sum of absolute deviations between the observed responses and the predictions produced by the model.
\begin{equation}
\min_{p} \sum_{i=1}^m \left| p(\mathbf{x}_i) - y_i \right|.
\end{equation}

Imposing a convexity constraint on the regressor $p(\mathbf{x})$ guarantees that the estimated model maintains the structural characteristics of the true function $f$. 
Directly enforcing convexity, however, leads to an optimization problem that is computationally intractable. 
To overcome this difficulty, prior work has introduced tractable relaxations of the convexity constraint \cite{hildreth1954point, magnani2005tractable, curmei2025shape}. 
A commonly adopted strategy replaces strict convexity with a sum-of-squares (SOS) convexity condition, 
allowing the regression task to be solved efficiently via semidefinite programming. 
More recent research has considered alternative formulations, such as convex SPLD polynomials \cite{jiao2026spld}, 
which are especially effective when the number of samples is large and the input dimension is relatively small.

Within this general framework, one often encounters settings in which the unknown function 
$f$ is not only convex but also exhibits low-order interactions among its variables. 
Under these structural conditions, a natural strategy is to approximate $f$ with 
a convex SPBQ polynomial. Constraining the hypothesis class to SPBQ forms, where the biquadratic component has size $n\times 2$, provides 
substantial computational benefits and greatly enhances the scalability of 
the resulting regression problem.

We pose the SPBQ convex regression problem as follows:
\begin{equation}\label{spbq cons}
\begin{aligned}
\min_{p} \quad & \sum_{i=1}^m \left| p(\mathbf{x}_i) - y_i \right| \\
\text{subject to} \quad 
& p \text{ is an SPBQ polynomial }, \\
& H(\mathbf{x}) \succeq 0 .
\end{aligned}
\end{equation}

In this formulation, optimization is carried out on the coefficients of the 
polynomial $p(\mathbf{x})$, while degree $2d$ is treated as a fixed parameter. 
Here, $H(\mathbf{x})$ denotes the Hessian matrix of $p(\mathbf{x})$, and the constraint 
$H(\mathbf{x}) \succeq 0$ ensures that the resulting estimator remains convex throughout its domain.

To evaluate the performance of our approach, we adopt the numerical experimental setup introduced in \cite{ahmadi2023sums}. Following their methodology, we consider a family of convex functions defined by:
\begin{equation}
f_{a,b}(\mathbf{x}) =
\log\!\left(\sum_{i=1}^{n} a_i e^{b_i x_i}\right)
+ \mathbf{x}^T (aa^T + I)\mathbf{x}
+ b^T \mathbf{x},
\end{equation}
where the parameter vectors $a, b \in \mathbb{R}^n$ are generated with entries drawn independently and uniformly from the intervals $[0,4]$ and $[-2,2]$, respectively, and $I$ denotes the identity matrix of $n \times n$.

In our simulations, we generate $r = 100$ independent instances of these functions in dimension $n = 10$. For each instance, we construct a training dataset of size $m = 300$ and a separate test dataset of size $t = 100$. The feature vectors $\mathbf{x}_i \in \mathbb{R}^{10}$ are sampled independently of the standard multivariate normal distribution. The corresponding response variables are computed as 
\begin{equation}
y_i = f_{a,b}(\mathbf{x}_i) + \epsilon_i,
\end{equation}
where the additive noise $\epsilon_i$ is extracted independently from a standard normal distribution.

Restricting the search space to polynomials of degree $2d = 4$, we compare the performance of three regression formulations summarized in Table \ref{tab:reg_models}.

\begin{table}[h]
\centering
\begin{tabular}{|p{4cm}|p{9cm}|}
\hline
\textbf{Regression Model} & \textbf{Optimization Problem Formulation} \\
\hline

Polynomial regression &
$\displaystyle \min_{p} \sum_{i=1}^{m} |p(\mathbf{x}_i) - y_i|$ \\
\hline

SOS-convex polynomial regression &
$\displaystyle 
\min_{p} \sum_{i=1}^{m} |p(\mathbf{x}_i) - y_i|
\quad \text{s.t.} \quad
H(\mathbf{x}) \text{ is an SOS matrix}
$ \\
\hline

SPBQ convex polynomial regression &
$\displaystyle
\min_{p } \sum_{i=1}^{m} |p(\mathbf{x}_i) - y_i|
\quad \text{s.t.} \quad
H(\mathbf{x}) \succeq 0\eqref{spbq cons}
$ \\
\hline
\end{tabular}
\caption{Regression models evaluated in the numerical experiments.}
\label{tab:reg_models}
\end{table}

Each regression problem is solved using the training dataset to obtain the optimal polynomial estimator, denoted by $p^*$. To assess generalization performance, we compute the average and maximum absolute deviations over the test data set.
\begin{equation}
\text{AvgDev}_{a,b}
= \frac{1}{t} \sum_{i=1}^{t}
\left| f_{a,b}(\mathbf{x}_i) - p^*(\mathbf{x}_i) \right|,
\end{equation}
\begin{equation}
\text{MaxDev}_{a,b}
= \max_{1 \le i \le t}
\left| f_{a,b}(\mathbf{x}_i) - p^*(\mathbf{x}_i) \right|.
\end{equation}

As illustrated in Figure \ref{fig:histograms}, standard unconstrained polynomial regression suffers from severe overfitting, producing test errors several orders of magnitude higher than the structured approaches. This extreme divergence is mathematically expected, as the number of polynomial coefficients for $n=10$ and $2d=4$ is $\binom{10+4}{4} = 1001$, which constitutes a highly under determined system given only training samples $m=300$. Because the performance gap spans multiple orders of magnitude, the resulting histograms are plotted on a logarithmic scale. Consistent with the findings in \cite{ahmadi2023sums}, the comparison of these test errors across the $100$ random instances clearly demonstrates that enforcing the SPBQ convexity constraint heavily mitigates this overfitting, allowing the SPBQ convex polynomial regression to achieve significantly lower approximation errors than both unconstrained and standard SOS-convex approaches.

\begin{figure}[h]
    \centering
     \includegraphics[width=\textwidth]{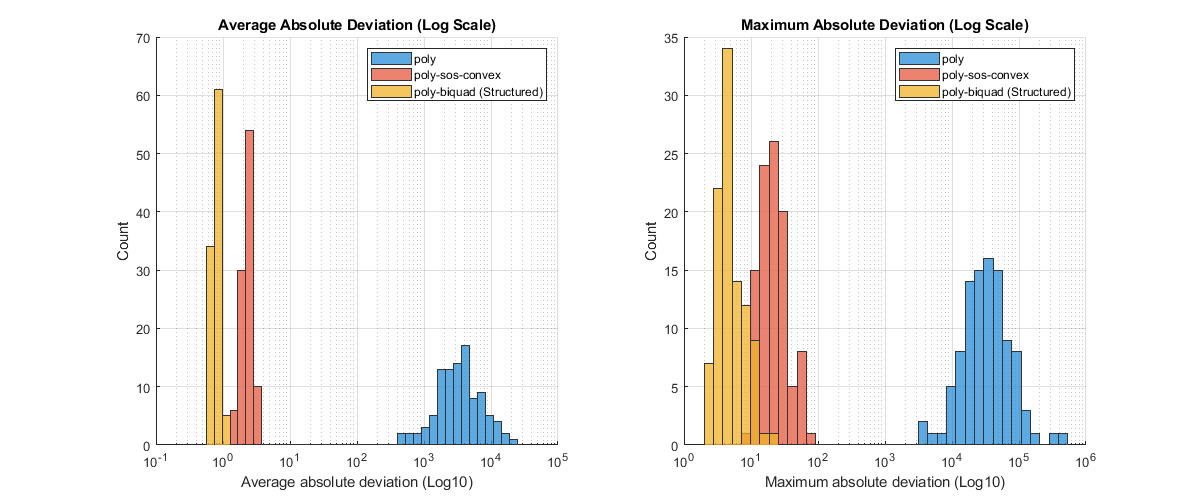}
    \caption{Histograms comparing the mean (left) and maximum (right) absolute deviations of the three regression models across $100$ randomly generated problem instances. The x-axes are displayed on a base-10 logarithmic scale to accommodate the pronounced overfitting behavior observed in the unconstrained polynomial regression model.}
    \label{fig:histograms}
\end{figure}

\section{Application to Fluid Dynamics}\label{fluid app.}
\label{sec:fluid_dynamics}

Investigations of fluid-flow stability and the establishment of bounds on turbulent-flow properties rely extensively on the incompressible Navier–Stokes equations. Because these partial differential equations are infinite-dimensional, a common analytical strategy is to expand the velocity field in a finite-dimensional Galerkin basis composed of orthogonal, divergence-free vector fields. This expansion transforms the original governing equations into a finite system of nonlinear ordinary differential equations of the form
\[
\dot{a}_i = f_i(a) = \frac{1}{Re} \sum_{j} L_{ij} a_j + \sum_{j,k} N_{ijk} a_j a_k,
\]
 where $a \in \mathbb{R}^n$ denotes the amplitudes of the Galerkin modes, $Re$ is the Reynolds number, $L$ characterizes viscous dissipation, and $N$ captures the nonlinear, energy-preserving convective transport. Traditionally, certifying global stability or obtaining rigorous bounds on time-averaged quantities—such as the turbulent energy dissipation rate $\Phi(a)$—relies on constructing suitable Lyapunov or auxiliary functions $V(a)$. Establishing global non-negativity of such polynomial functions is known to be NP-hard. The sum-of-squares (SOS) relaxation has emerged as an effective approach to circumvent this difficulty by replacing the non-negativity constraint with the requirement that the polynomial admits an SOS decomposition, which can be computed efficiently using semidefinite programming (SDP) \cite{chernyshenko2014polynomial}.

In practice, however, standard SOS techniques face a severe computational bottleneck in fluid dynamical applications. The usual formulation searches over a fully dense polynomial basis. For an $n$-dimensional system and a candidate polynomial of degree $d$, the size of this basis grows combinatorially as $\binom{n + d/2}{d/2}$. As the order of the Galerkin truncation is increased to resolve realistic turbulent dynamics, conventional SOS formulations rapidly deplete available memory. To address this challenge, we restrict the search to the SPBQ class, which overcomes the curse of dimensionality while still yielding mathematically rigorous bounds.

We first illustrate the effectiveness of the SPBQ class by computing the global stability envelope of a well-known $9$-dimensional Galerkin model of Couette flow \cite{moehlis2004low}. For a candidate Lyapunov function $V(a)$, global stability requires that $V(a)$ be positive definite and that its time derivative,
\[
\dot{V}(a) = \nabla V(a) \cdot f(a)
\]
is strictly negative definite. Because the vector field $f(a)$ possesses a quadratic nonlinearity, an arbitrary degree-4 function will yield a degree-5 time derivative, which cannot be globally negative definite. To construct a valid SOS constraint, we exploit the energy-conserving property of the fluid nonlinearities and propose a candidate Lyapunov function dominated by our SPBQ structure, supported by lower-degree polynomial buffers. 

\[
V(a) = s(x, y) + b(x, y) + p_3(a) + p_2(a)
\]
Here, the state vector $a$ is partitioned into two distinct sets of physical modes $x \in \mathbb{R}^7$ and $y \in \mathbb{R}^2$, and $s(x,y) + b(x,y)$ constitutes the homogeneous form of the degree-4 SPBQ. 

By conducting a bisection search over the Reynolds number, we evaluated the stability envelope. The classical energy stability limit for this model is $Re_E = 7.5$, while previous studies utilizing a dense full SOS approach achieved a certified maximum bound of $Re_{SOS} = 54.1$ \cite{goulart2012global}. The structured SPBQ formulation successfully certifies global stability up to $Re \approx 12.50$. This confirms that the highly restricted $(7,2)$ bipartite structure retains sufficient degrees of freedom to capture critical nonlinear energy transfers and strictly improve the bounds of standard linear energy stability.

Although the stability envelope demonstrates the theoretical validity of the SPBQ structure, its most significant advantage lies in computational scalability. We benchmark this scalability through the calculation of turbulent energy dissipation bounds. Following the background-flow methodology, finding a rigorous lower bound $C$ on the time-averaged dissipation rate $\overline{\Phi}$ is formulated as finding a function $V(a)$ such that 
\[D(a) = \dot{V}(a) + \Phi(a) - C \ge 0
\]
By maximizing $C$ subject to $D(a)$ being SOS, one obtains the optimal bound. We conducted a scaling benchmark comparing the standard dense degree-4 SOS formulation, using native compiler tools, against the SPBQ formulation utilizing our direct LMI construction. The experiments were conducted on Galerkin projections ranging from $n=10$ to $n=30$. The SPBQ method enforced the partition $y \in \mathbb{R}^2$ and $x \in \mathbb{R}^{n-2}$, effectively neutralizing the combinatorial explosion of the Gram matrix.

\begin{table}[ht]
\centering
\begin{tabular}{|c | c | c | c |}
\hline\hline
$n$ & Standard SOS Time(s) & SPBQ Time(s) & Speedup \\
\hline
10 & 5.22 & 0.73 & 7.2$\times$ \\
15 & 18.69 & 0.36 & 52.5$\times$ \\
20 & Out of Memory & 0.70 & --- \\
25 & Out of Memory & 1.37 & --- \\
30 & Out of Memory & 2.65 & --- \\
\hline\hline
\end{tabular}
\caption{Computational scaling of the turbulent energy dissipation bound: standard SOS runs out of memory for $n \ge 20$, whereas the SPBQ direct LMI scales efficiently to high-dimensional models.}
\label{tab:scaling_benchmark_fluid}
\end{table}

As reported in Table \ref{tab:scaling_benchmark_fluid}, the computational benefits grow exponentially. For a 15-mode system, the SPBQ formulation delivers a high speed relative to the conventional method. More importantly, once the state dimension reaches $n=20$ and higher, constructing the standard dense degree-4 Gram matrix fully depletes system RAM, causing the solver to fail. In sharp contrast, by avoiding symbolic polynomial parsing and instead employing the explicit Direct LMI mapping, the SPBQ approach is able to compute the dissipation bound for a 30-mode system in under three seconds. This level of computational efficiency demonstrates that the SPBQ family offers a highly tractable framework for polynomial optimization in high-dimensional fluid dynamics problems.
\section{Conclusion}

In this work, we introduced and rigorously defined Sum of Separable and Biquadratic (SPBQ) polynomials, thereby establishing a systematic framework for analysing nonnegativity and convexity in specialised quartic forms. We derived a precise theoretical characterisation by proving that, for SPBQ forms whose biquadratic part has dimension $n \times 2$, convexity and SOS-convexity coincide. We then constructed an explicit $3 \times 3$ counterexample to pinpoint the boundary of this equivalence, showing that in higher dimensions convex SPBQ forms are not necessarily SOS-convex.

In addition to these theoretical insights, we showed that the SPBQ framework offers substantial computational benefits. Leveraging the structural properties of SPBQ polynomials, we can drastically reduce the size of the required semidefinite constraints. This reduction translates into major computational speedups, enabling the solution of large-scale optimisation problems while overcoming the memory bottlenecks that typically hinder more general SOS formulations.

Taken together, our results support efficient global optimisation over a novel class of quartic polynomials, paving the way for scalable and promising applications in fields such as convex polynomial regression, fluid dynamics and others.

\bibliographystyle{apalike} 
\bibliography{reference}

\newpage
\section*{Appendix A}\label{Appendix:A}

Using the structural symmetries \cite{parrilo2005exploiting, gatermann2004symmetry} of the polynomial \eqref{counter}, the entire block of $286 \times 286$ Gram matrix diagonalizes into 142 independent components. The numerical evaluation reveals that the rational coefficients corresponding to the 140 isolated $1 \times 1$ blocks are exactly 0. Consequently, the SOS decomposition is completely simplified to two dense $73 \times 73$ blocks:
\begin{equation}
(x_1^2+x_2^2 +x_3^2 + y_1^2+y_2^2 + y_3^2)z^TH(x,y)z = z_A^T Q_A z_A + z_B^T Q_B z_B,
\end{equation}
where $Q_A, Q_B \in \mathbb{Q}^{73 \times 73}$ are rational symmetric matrices, and $z_A, z_B$ are the symmetry-adapted monomial basis vectors given by:
\begin{align*}
z_A = [& z_3 z_5 z_6,\; z_3 z_4 z_6,\; z_3 z_4 z_5,\; z_2 z_5 z_6,\; z_2 z_4 z_6,\; z_2 z_4 z_5,\; z_1 z_5 z_6,\; z_1 z_4 z_6,\; z_1 z_4 z_5,\; z_1 z_2 z_3,\; y_3^2 z_3,\; y_3^2 z_2,\\
& y_3^2 z_1,\; y_2 y_3 z_3,\; y_2 y_3 z_2,\; y_2 y_3 z_1,\; y_2^2 z_3,\; y_2^2 z_2,\; y_2^2 z_1,\; y_1 y_3 z_3,\; y_1 y_3 z_2,\; y_1 y_3 z_1,\; y_1 y_2 z_3,\; y_1 y_2 z_2,\\
& y_1 y_2 z_1,\; y_1^2 z_3,\; y_1^2 z_2,\; y_1^2 z_1,\; x_3 y_3 z_6,\; x_3 y_3 z_5,\; x_3 y_3 z_4,\; x_3 y_2 z_6,\; x_3 y_2 z_5,\; x_3 y_2 z_4,\; x_3 y_1 z_6,\; x_3 y_1 z_5,\\
& x_3 y_1 z_4,\; x_3^2 z_3,\; x_3^2 z_2,\; x_3^2 z_1,\; x_2 y_3 z_6,\; x_2 y_3 z_5,\; x_2 y_3 z_4,\; x_2 y_2 z_6,\; x_2 y_2 z_5,\; x_2 y_2 z_4,\; x_2 y_1 z_6,\; x_2 y_1 z_5,\\
& x_2 y_1 z_4,\; x_2 x_3 z_3,\; x_2 x_3 z_2,\; x_2 x_3 z_1,\; x_2^2 z_3,\; x_2^2 z_2,\; x_2^2 z_1,\; x_1 y_3 z_6,\; x_1 y_3 z_5,\; x_1 y_3 z_4,\; x_1 y_2 z_6,\; x_1 y_2 z_5,\\
& x_1 y_2 z_4,\; x_1 y_1 z_6,\; x_1 y_1 z_5,\; x_1 y_1 z_4,\; x_1 x_3 z_3,\; x_1 x_3 z_2,\; x_1 x_3 z_1,\; x_1 x_2 z_3,\; x_1 x_2 z_2,\; x_1 x_2 z_1,\; x_1^2 z_3,\; x_1^2 z_2,\\
&  x_1^2 z_1]^T
\end{align*}
\begin{align*}
    z_B = [& z_4 z_5 z_6,\; z_2 z_3 z_6,\; z_2 z_3 z_5,\; z_2 z_3 z_4,\; z_1 z_3 z_6,\; z_1 z_3 z_5,\; z_1 z_3 z_4,\; z_1 z_2 z_6,\; z_1 z_2 z_5,\; z_1 z_2 z_4,\; y_3^2 z_6,\; y_3^2 z_5,\\
    & y_3^2 z_4,\; y_2 y_3 z_6,\; y_2 y_3 z_5,\; y_2 y_3 z_4,\; y_2^2 z_6,\; y_2^2 z_5,\; y_2^2 z_4,\; y_1 y_3 z_6,\; y_1 y_3 z_5,\; y_1 y_3 z_4,\; y_1 y_2 z_6,\; y_1 y_2 z_5,\\
    & y_1 y_2 z_4,\; y_1^2 z_6,\; y_1^2 z_5,\; y_1^2 z_4,\; x_3 y_3 z_3,\; x_3 y_3 z_2,\; x_3 y_3 z_1,\; x_3 y_2 z_3,\; x_3 y_2 z_2,\; x_3 y_2 z_1,\; x_3 y_1 z_3,\; x_3 y_1 z_2,\\
    & x_3 y_1 z_1,\; x_3^2 z_6,\; x_3^2 z_5,\; x_3^2 z_4,\; x_2 y_3 z_3,\; x_2 y_3 z_2,\; x_2 y_3 z_1,\; x_2 y_2 z_3,\; x_2 y_2 z_2,\; x_2 y_2 z_1,\; x_2 y_1 z_3,\; x_2 y_1 z_2,\\
    & x_2 y_1 z_1,\; x_2 x_3 z_6,\; x_2 x_3 z_5,\; x_2 x_3 z_4,\; x_2^2 z_6,\; x_2^2 z_5,\; x_2^2 z_4,\; x_1 y_3 z_3,\; x_1 y_3 z_2,\; x_1 y_3 z_1,\; x_1 y_2 z_3,\; x_1 y_2 z_2,\\
    & x_1 y_2 z_1,\; x_1 y_1 z_3,\; x_1 y_1 z_2,\; x_1 y_1 z_1,\; x_1 x_3 z_6,\; x_1 x_3 z_5,\; x_1 x_3 z_4,\; x_1 x_2 z_6,\; x_1 x_2 z_5,\; x_1 x_2 z_4,\; x_1^2 z_6,\; x_1^2 z_5,\\
    & x_1^2 z_4]^T.
\end{align*}

To rigorously certify that $Q_A \succeq 0$ and $Q_B \succeq 0$, we compute the exact rational $LDL^T$ decompositions:
\begin{equation}
Q_A = L_A D_A L_A^T, \quad Q_B = L_B D_B L_B^T.
\end{equation}

Due to their dimensionality $73 \times 73$, the exact rational matrices $L_A, D_A$ and $L_B, D_B$ are too large to be typed in the print. However, we verify that every element of the exact rational diagonal matrices $D_A$ and $D_B$ is strictly positive\footnote{\href{https://colab.research.google.com/drive/1cx1uTwxm6uVLFumnlWa1yuK71x3BMLEu\#scrollTo=eVVM-Uc0xZ8g}{Rational decomposition}}. This provides an exact analytical proof that $Q_A \succeq 0$ and $Q_B \succeq 0$, globally certifying the convexity of $f$.

\section*{Appendix B}\label{Appendix:B}

\begin{align*}
z^T \nabla^2 f z =\;
& 0.2\Big[(y_1 - y_2)z_{x_1} + x_1(z_{y_1} - z_{y_2})\Big]^2 \\
&+ 0.2\Big[(y_2 - y_3)z_{x_2} + x_2(z_{y_2} - z_{y_3})\Big]^2 \\
&+ 0.2\Big[(y_3 - y_1)z_{x_3} + x_3(z_{y_3} - z_{y_1})\Big]^2 \\
&+ 0.2\Big[(x_1 - x_2)z_{y_1} + y_1(z_{x_1} - z_{x_2})\Big]^2 \\
&+ 0.2\Big[(x_2 - x_3)z_{y_2} + y_2(z_{x_2} - z_{x_3})\Big]^2 \\
&+ 0.2\Big[(x_3 - x_1)z_{y_3} + y_3(z_{x_3} - z_{x_1})\Big]^2 \\
&+ 0.4\Big[x_1z_{x_1} + y_2z_{y_2}\Big]^2
+ 0.4\Big[x_2z_{x_2} + y_3z_{y_3}\Big]^2
+ 0.4\Big[x_3z_{x_3} + y_1z_{y_1}\Big]^2 \\
&+ 0.4\Big[x_2z_{x_2} + y_1z_{y_1}\Big]^2
+ 0.4\Big[x_3z_{x_3} + y_2z_{y_2}\Big]^2
+ 0.4\Big[x_1z_{x_1} + y_3z_{y_3}\Big]^2 \\
&+ 0.4\Big[x_1z_{x_1} + y_2z_{y_2}\Big]^2
+ 0.4\Big[x_2z_{x_2} + y_3z_{y_3}\Big]^2
+ 0.4\Big[x_3z_{x_3} + y_1z_{y_1}\Big]^2 \\
&+ 0.4\Big[x_2z_{x_2} + y_1z_{y_1}\Big]^2
+ 0.4\Big[x_3z_{x_3} + y_2z_{y_2}\Big]^2
+ 0.4\Big[x_1z_{x_1} + y_3z_{y_3}\Big]^2 \\
&+ 0.2\Big[x_1z_{x_1} + y_1z_{y_1}\Big]^2
+ 0.2\Big[x_1z_{x_1} + y_2z_{y_2}\Big]^2
+ 0.2\Big[x_2z_{x_2} + y_2z_{y_2}\Big]^2 \\
&+ 0.2\Big[x_2z_{x_2} + y_3z_{y_3}\Big]^2
+ 0.2\Big[x_3z_{x_3} + y_3z_{y_3}\Big]^2
+ 0.2\Big[x_3z_{x_3} + y_1z_{y_1}\Big]^2 \\
&+ 0.2\Big[x_1z_{x_1} + y_1z_{y_1}\Big]^2
+ 0.2\Big[x_2z_{x_2} + y_1z_{y_1}\Big]^2
+ 0.2\Big[x_2z_{x_2} + y_2z_{y_2}\Big]^2 \\
&+ 0.2\Big[x_3z_{x_3} + y_2z_{y_2}\Big]^2
+ 0.2\Big[x_3z_{x_3} + y_3z_{y_3}\Big]^2
+ 0.2\Big[x_1z_{x_1} + y_3z_{y_3}\Big]^2 \\
&+ 0.4y_2^2z_{x_1}^2 + 0.4x_1^2z_{y_2}^2
+ 0.4y_3^2z_{x_2}^2 + 0.4x_2^2z_{y_3}^2 \\
&+ 0.4y_1^2z_{x_3}^2 + 0.4x_3^2z_{y_1}^2
+ 0.4y_1^2z_{x_2}^2 + 0.4x_2^2z_{y_1}^2 \\
&+ 0.4y_2^2z_{x_3}^2 + 0.4x_3^2z_{y_2}^2
+ 0.4y_3^2z_{x_1}^2 + 0.4x_1^2z_{y_3}^2 \\
&+ 9.6\Big(x_1^2z_{x_1}^2 + x_2^2z_{x_2}^2 + x_3^2z_{x_3}^2
+ y_1^2z_{y_1}^2 + y_2^2z_{y_2}^2 + y_3^2z_{y_3}^2\Big) \\
&- 0.4x_1y_2z_{x_1}z_{y_1}
- 0.4x_1y_1z_{x_1}z_{y_2}
- 0.4x_2y_3z_{x_2}z_{y_2}
- 0.4x_2y_2z_{x_2}z_{y_3} \\
&- 0.4x_3y_1z_{x_3}z_{y_3}
- 0.4x_3y_3z_{x_3}z_{y_1}
- 0.4x_1y_1z_{x_2}z_{y_1}
- 0.4x_2y_1z_{x_1}z_{y_1} \\
&- 0.4x_2y_2z_{x_3}z_{y_2}
- 0.4x_3y_2z_{x_2}z_{y_2}
- 0.4x_3y_3z_{x_1}z_{y_3}
- 0.4x_1y_3z_{x_3}z_{y_3}.
\end{align*}

\end{document}